\documentclass[10pt,reqno]{amsart}
\usepackage{latexsym,amsmath,amssymb,amscd}
\usepackage[all]{xy}

\def\today{\ifcase \month \or
   January \or February \or March \or April \or
   May \or June \or July \or August \or
   September \or October \or November \or December \fi
   \space\number\day , \number\year}
   
\makeatletter
  \newcommand\@dotsep{4.5}
  \def\@tocline#1#2#3#4#5#6#7{\relax
     \ifnum #1>\c@tocdepth 
     \else
     \par \addpenalty\@secpenalty\addvspace{#2}%
     \begingroup \hyphenpenalty\@M
     \@ifempty{#4}{%
     \@tempdima\csname r@tocindent\number#1\endcsname\relax
        }{%
         \@tempdima#4\relax
           }%
      \parindent\z@ \leftskip#3\relax \advance\leftskip\@tempdima\relax
      \rightskip\@pnumwidth plus1em \parfillskip-\@pnumwidth
       #5\leavevmode\hskip-\@tempdima #6\relax
       \leaders\hbox{$\m@th
       \mkern \@dotsep mu\hbox{.}\mkern \@dotsep mu$}\hfill
       \hbox to\@pnumwidth{\@tocpagenum{#7}}\par
       \nobreak
        \endgroup
         \fi}
\makeatother 

\begin{document}


\makeatletter
\@addtoreset{figure}{section}
\def\thefigure{\thesection.\@arabic\c@figure}
\def\fps@figure{h,t}
\@addtoreset{table}{bsection}

\def\thetable{\thesection.\@arabic\c@table}
\def\fps@table{h, t}
\@addtoreset{equation}{section}
\def\theequation{
\arabic{equation}}
\makeatother

\newcommand{\bfi}{\bfseries\itshape}

\newtheorem{theorem}{Theorem}
\newtheorem{acknowledgment}[theorem]{Acknowledgment}
\newtheorem{corollary}[theorem]{Corollary}
\newtheorem{definition}[theorem]{Definition}
\newtheorem{example}[theorem]{Example}
\newtheorem{examples}[theorem]{Examples}
\newtheorem{lemma}[theorem]{Lemma}
\newtheorem{notation}[theorem]{Notation}
\newtheorem{proposition}[theorem]{Proposition}
\newtheorem{remark}[theorem]{Remark}
\newtheorem{setting}[theorem]{Setting}

\numberwithin{theorem}{section}
\numberwithin{equation}{section}

\renewcommand{\1}{{\bf 1}}
\newcommand{\Ad}{{\rm Ad}}
\newcommand{\Alg}{{\rm Alg}\,}
\newcommand{\alg}{*{\rm alg}}
\newcommand{\Aut}{{\rm Aut}\,}
\newcommand{\ad}{{\rm ad}}
\newcommand{\Borel}{{\rm Borel}}
\newcommand{\botimes}{\bar\otimes}
\newcommand{\Ci}{{\mathcal C}^\infty}
\newcommand{\Cint}{{\mathcal C}^\infty_{\rm int}}
\newcommand{\Cpol}{{\mathcal C}^\infty_{\rm pol}}
\newcommand{\cont}{{\rm cont}}
\newcommand{\Der}{{\rm Der}\,}
\newcommand{\de}{{\rm d}}
\newcommand{\ee}{{\rm e}}
\newcommand{\End}{{\rm End}\,}
\newcommand{\ev}{{\rm ev}}
\newcommand{\id}{{\rm id}}
\newcommand{\ie}{{\rm i}}
\newcommand{\ind}{\mathop{\rm ind}}
\newcommand{\GL}{{\rm GL}}
\newcommand{\gl}{{{\mathfrak g}{\mathfrak l}}}
\newcommand{\Hom}{{\rm Hom}\,}
\newcommand{\Img}{{\rm Im}\,}
\newcommand{\Ind}{{\rm Ind}}
\newcommand{\Ker}{{\rm Ker}\,}
\newcommand{\Lie}{\text{\bf L}}
\newcommand{\lin}{{\rm lin}}
\newcommand{\Mt}{{{\mathcal M}_{\text t}}}
\newcommand{\m}{\text{\bf m}}
\newcommand{\pr}{{\rm pr}}
\newcommand{\Ran}{{\rm Ran}\,}
\renewcommand{\Re}{{\rm Re}\,}
\newcommand{\so}{\text{so}}
\newcommand{\spa}{{\rm span}\,}
\newcommand{\Tr}{{\rm Tr}\,}
\newcommand{\tw}{\ast_{\rm tw}}
\newcommand{\Op}{{\rm Op}}
\newcommand{\U}{{\rm U}}
\newcommand{\weak}{\text{weak}}

\newcommand{\QuadrHilb}{\textbf{QuadrHilb}}
\newcommand{\LieGr}{\textbf{LieGr}}

\newcommand{\UC}{{{\mathcal U}{\mathcal C}}}
\newcommand{\UCb}{{{\mathcal U}{\mathcal C}_b}}
\newcommand{\LUCb}{{{\mathcal L}{\mathcal U}{\mathcal C}_b}}
\newcommand{\RUCb}{{{\mathcal R}{\mathcal U}{\mathcal C}_b}}
\newcommand{\LUCl}{{{\mathcal L}{\mathcal U}{\mathcal C}}_{\rm loc}}
\newcommand{\RUCl}{{{\mathcal R}{\mathcal U}{\mathcal C}}_{\rm loc}}
\newcommand{\UCl}{{{\mathcal U}{\mathcal C}}_{\rm loc}}

\newcommand{\CC}{{\mathbb C}}
\newcommand{\RR}{{\mathbb R}}
\newcommand{\TT}{{\mathbb T}}

\newcommand{\Ac}{{\mathcal A}}
\newcommand{\Bc}{{\mathcal B}}
\newcommand{\Cc}{{\mathcal C}}
\newcommand{\Dc}{{\mathcal D}}
\newcommand{\Ec}{{\mathcal E}}
\newcommand{\Fc}{{\mathcal F}}
\newcommand{\Hc}{{\mathcal H}}
\newcommand{\Jc}{{\mathcal J}}
\newcommand{\Lc}{{\mathcal L}}
\renewcommand{\Mc}{{\mathcal M}}
\newcommand{\Nc}{{\mathcal N}}
\newcommand{\Oc}{{\mathcal O}}
\newcommand{\Pc}{{\mathcal P}}
\newcommand{\Qc}{{\mathcal Q}}
\newcommand{\Rac}{{\mathcal R}}
\newcommand{\Sc}{{\mathcal S}}
\newcommand{\Tc}{{\mathcal T}}
\newcommand{\Uc}{{\mathcal U}}
\newcommand{\Vc}{{\mathcal V}}
\newcommand{\Wig}{{\mathcal W}}
\newcommand{\Xc}{{\mathcal X}}
\newcommand{\Yc}{{\mathcal Y}}

\newcommand{\Bg}{{\mathfrak B}}
\newcommand{\Dg}{{\mathfrak D}}
\newcommand{\Fg}{{\mathfrak F}}
\newcommand{\Gg}{{\mathfrak G}}
\newcommand{\Ig}{{\mathfrak I}}
\newcommand{\Jg}{{\mathfrak J}}
\newcommand{\Lg}{{\mathfrak L}}
\newcommand{\Pg}{{\mathfrak P}}
\newcommand{\Rg}{{\mathfrak R}}
\newcommand{\Sg}{{\mathfrak S}}
\newcommand{\Xg}{{\mathfrak X}}
\newcommand{\Yg}{{\mathfrak Y}}
\newcommand{\Zg}{{\mathfrak Z}}

\newcommand{\ag}{{\mathfrak a}}
\newcommand{\bg}{{\mathfrak b}}
\newcommand{\dg}{{\mathfrak d}}
\renewcommand{\gg}{{\mathfrak g}}
\newcommand{\hg}{{\mathfrak h}}
\newcommand{\kg}{{\mathfrak k}}
\newcommand{\mg}{{\mathfrak m}}
\newcommand{\n}{{\mathfrak n}}
\newcommand{\og}{{\mathfrak o}}
\newcommand{\pg}{{\mathfrak p}}
\newcommand{\sg}{{\mathfrak s}}
\newcommand{\tg}{{\mathfrak t}}
\newcommand{\ug}{{\mathfrak u}}
\newcommand{\zg}{{\mathfrak z}}

\newcommand{\ZZ}{\mathbb Z}
\newcommand{\NN}{\mathbb N}
\newcommand{\BB}{\mathbb B}
\newcommand{\HH}{\mathbb H}

\newcommand{\ep}{\varepsilon}

\newcommand{\hake}[1]{\langle #1 \rangle }

\newcommand{\scalar}[2]{\langle #1 ,#2 \rangle }
\newcommand{\vect}[2]{(#1_1 ,\ldots ,#1_{#2})}
\newcommand{\norm}[1]{\Vert #1 \Vert }
\newcommand{\normrum}[2]{{\norm {#1}}_{#2}}

\newcommand{\upp}[1]{^{(#1)}}
\newcommand{\p}{\partial}

\newcommand{\opn}{\operatorname}
\newcommand{\slim}{\operatornamewithlimits{s-lim\,}}
\newcommand{\sgn}{\operatorname{sgn}}

\newcommand{\seq}[2]{#1_1 ,\dots ,#1_{#2} }
\newcommand{\loc}{_{\opn{loc}}}

\makeatletter
\title[On differentiability of vectors]{On differentiability of vectors in  Lie group representations}
\author{Ingrid Belti\c t\u a 
 and Daniel Belti\c t\u a
}
\address{Institute of Mathematics ``Simion Stoilow'' 
of the Romanian Academy, 
P.O. Box 1-764, Bucharest, Romania}
\email{Ingrid.Beltita@imar.ro}
\email{Daniel.Beltita@imar.ro}
\keywords{Lie group; topological group; unitary representation; smooth vector}
\subjclass[2000]{Primary 22E65; Secondary 22E66,22A10,22A25}
\date{\today}
\makeatother

\begin{abstract}
We address a linearity problem for differentiable vectors in representations of infinite-dimensional Lie groups on locally convex spaces, 
which is similar to the linearity problem for the directional derivatives of functions.
\end{abstract}

\maketitle

\section{Introduction}

In this paper we address an issue raised in \cite[Probl. 13.4]{Nee10} concerning some linearity properties for differentiable vectors in representations 
of infinite-dimensional Lie groups; see also \cite{Nee11} and the concluding comments in Examples~\ref{cont_ex} below. 
Before proceeding to a more detailed description of our approach to that question, we should mention that the recent interest in this circle of ideas is motivated in part by some applications of representation theory in the theory of operator algebras on the one hand (see \cite{Nee10}) and in the theory of partial differential equations on the other hand (see \cite{BB09} and \cite{BB11}). 

It is worthwhile to place ourselves for the moment in the setting of topological groups, although the main classes of examples are provided by infinite-dimensional Lie groups modeled on locally convex spaces 
(see situations \eqref{lin_ex_item3}--\eqref{lin_ex_item3.6} and \eqref{lin_ex_item3.4}, 
and also sometimes \eqref{lin_ex_item1}, in Examples~\ref{lin_ex} below). 
Thus let $G$ be a topological group and denote by $\Lg(G)$ the set of its one-parameter subgroups (i.e., continuous homomorphisms from the additive group $\RR$ into $G$), endowed with topology of uniform convergence on the compact subsets of~$\RR$. 
For $X,X_1,X_2\in\Lg(G)$  we say that $X=X_1+X_2$ whenever we have     
\begin{equation}\label{aux8_eq1}
(\forall t\in\RR)\quad 
X(t)=\lim\limits_{n\to\infty}
\Bigl(X_1\Bigl(\frac{t}{n}\Bigr)X_2\Bigl(\frac{t}{n}\Bigr)\Bigr)^n
\end{equation}
uniformly on every compact subset of $\RR$. 

Next let $\pi\colon G\to\End(\Yc)$ be a representation on 
a locally convex space~$\Yc$ such that the group action 
$G\times\Yc\to\Yc$, $(g,y)\mapsto\pi(g)y$ 
is continuous.
(The continuous unitary representations provide a rich class of interesting and non-trivial examples, however we are also interested in representations on locally convex spaces in order to make our results applicable for group representations in spaces of smooth functions or smooth vectors.)
For every $X\in\Lg(G)$ we consider the infinitesimal generator of the one-parameter group of operators $\pi(X(\cdot))$, 
$$\de\pi(X):=\frac{\de}{\de t}\Big\vert_{t=0}\pi(X(t))\colon \Dc(\de\pi(X))\to\Yc.$$ 
Next define the following linear subspaces of $\Yc$: 
\begin{align}
\Dc_{\de\pi}^1
&:=\bigcap_{X\in\Lg(G)}\Dc(\de\pi(X)), \nonumber \\
\Dc_{\de\pi}^{1,\cont}
&:=\{y\in\Dc_{\de\pi}^1\mid \de\pi(\cdot)y\in\Cc(\Lg(G),\Yc)\} \nonumber \\
\Dc_{\de\pi}^{1,\lin}
&:=\{y\in\Dc_{\de\pi}^1\mid\text{if }X,X_1,X_2\in\Lg(G)\text{ and }X=X_1+X_2
\nonumber \\ 
&\hskip70pt \text{ then } 
\de\pi(X)y=\de\pi(X_1)y+\de\pi(X_2)y\}.\nonumber
\end{align}
Here and henceforth we denote by $\Cc(A,S)$ the space of continuous mappings between any 
topological spaces $A$ and~$S$. 

The main motivation for the present paper consists in understanding the relationship between these spaces. 
In particular, we are interested in describing wide classes of group representations for which the above three spaces coincide. 
It turned out in \cite{Nee10} that this problem is far from being trivial even in the special case when $G$ is a Banach-Lie group and $\Yc$ is a Banach space. 
By way of showing the difficulty of this problem, we recall that if $G$ is the group of unitary operators on some Hilbert space, endowed with the strong operator topology, then \eqref{aux8_eq1} is the Trotter formula for the sum of the infinitesimal generators of the corresponding one-parameter groups, and the addition of unbounded self-adjoint operators is a highly nontrivial problem, 
as discussed in \cite{Far75}; 
see also \cite{Che74} and \cite[Ex. 2.1]{Bel10} for a lot of pathologies related to this operation. 

Our approach to the aforementioned problem relies on a basic technique developed in \cite{Nel69} and \cite{BC73} that turns out to work in a more general setting; see Lemma~\ref{aux8} below. 
Our main results are Theorems \ref{lin_th} and \ref{final},  
which in particular provide an alternative proof to \cite[Th. 8.2]{Nee10}. 

Throughout the paper we assume that the topological groups and the locally convex spaces involved are Hausdorff spaces.

\section{Differentiable functions}

The main result of this section is Theorem~\ref{lin_th} which establishes the relationship between the above property~\eqref{aux8_eq1} and the ``first-order differential operators'' in a topological setting. 
In the first part of the following definition we recall some notions introduced in \cite{BC73} and \cite{BCR81}. 

\begin{definition}\label{aux4}
\normalfont
Let $G$ be a topological group and $\Yc$ be a locally convex space.  
\begin{enumerate}
\item\label{aux4_item1}  
If $\phi\colon G\to\Yc$, $X\in\Lg(G)$, and $g\in G$, then we denote 
\begin{equation}\label{aux4_eq1}
(D^\lambda_X\phi)(g)=\lim_{t\to 0}\frac{\phi(gX(t))-\phi(g)}{t} 
\end{equation}
whenever the limit in the right-hand side exists. 

\item\label{aux4_item2}  
We define $\Cc^1(G,\Yc)$ as the set of all $\phi\in\Cc(G,\Yc)$ 
such that the function 
$$
D^\lambda\phi\colon\Lg(G)\times G\to\Yc,\quad 
(D^\lambda\phi)(X,g)=(D^\lambda_X\phi)(g) $$
is well defined and continuous. 
We also denote $D^\lambda\phi=(D^\lambda)^{1}\phi$. 

Now let $n\ge 2$ and assume that the space $\Cc^{n-1}(G,\Yc)$ and the mapping 
$(D^\lambda)^{n-1}$ have been defined. 
Then we define $\Cc^n(G,\Yc)$ as the set of all functions $\phi\in\Cc^{(n-1)}(G,\Yc)$ 
such that the function 
$$\begin{aligned}
(D^\lambda)^{n}\phi\colon 
\Lg(G)\times\cdots\times\Lg(G)\times G & \to\Yc,\\
(X_1,\dots,X_n,g) & \mapsto (D^\lambda_{X_1}(D^\lambda_{X_2}\cdots(D^\lambda_{X_n}\phi)\cdots))(g)
\end{aligned} $$
is well defined and continuous. 

Moreover we define $\Ci(G,\Yc):=\bigcap\limits_{n\ge1}\Cc^n(G,\Yc)$. 

\item\label{aux4_item3} 
Let $\LUCl(G,\Yc)$ be the set of all functions $\phi\colon G\to\Yc$ such that every point $g_0\in G$ has a neighborhood $V$ such that $\phi\vert_V$ is left uniformly continuous, in the sense that for every neighborhood $U$ of $0\in\Yc$ there exists a neighborhood $W$ of $\1\in G$ such that if $x,y\in V$ and $x^{-1}y\in W$, then $\phi(x)-\phi(y)\in U$. 
We define $\LUCl^1(G,\Yc)$ as the set of all functions $\phi\in\LUCl(G,\Yc)$ 
such that the above mapping $D^\lambda\phi\colon\Lg(G)\times G\to\Yc$ 
is well defined and for every $X\in\Lg(G)$ we have $D^\lambda_X\phi\in\LUCl(G,\Yc)$. 
For $n\ge 2$, if the space $\LUCl^{n-1}(G,\Yc)$  has been defined,  
then we define $\LUCl^n(G,\Yc)$ as the set of all functions $\phi\in\LUCl^{n-1}(G,\Yc)$ 
such that the mapping $(D^\lambda)^n\phi$ is well defined 
and for every $X_1,\dots,X_n$ we have $D^\lambda_{X_1}(D^\lambda_{X_2}\cdots(D^\lambda_{X_n}\phi)\cdots)
\in\LUCl(G,\Yc)$. 
Moreover we define $\LUCl^\infty(G,\Yc):=\bigcap\limits_{n\ge1}\LUCl^n(G,\Yc)$.

\end{enumerate}
If $\Yc=\CC$, then we write simply $\Cc^n(G):=\Cc^n(G,\CC)$, $\LUCl^n(G):=\LUCl^n(G,\CC)$,  
etc.,  
for $n=1,2,\dots,\infty$.
\end{definition}

We now record some simple remarks on the notions introduced above.  
If $G$ is a locally compact group,
then $\Cc(G,\Yc)=\LUCl(G,\Yc)$, 
since continuous functions on compact sets are uniformly continuous. 
On the other hand, if $G$ is Banach-Lie group and $\Yc$ is a Banach space, then  
$\Cc^n(G,\Yc)\subseteq\LUCl^{n-1}(G,\Yc)$ for every $n\ge 1$, 
where $\LUCl^0(G,\Yc):=\LUCl(G,\Yc)$. 
Note that the $\Cc^n$-concept used here involves directional derivatives rather 
than the Fr\'echet differentials of functions 
as in the  
$\Cc^n$-concept habitually used in the differential calculus on Banach manifolds.  
They agree for finite-dimensional Lie groups 
(\cite[Folg. 1.4]{BC73}), 
however the one used here is slightly weaker in infinite dimensions. 
Nevertheless, if for instance $\phi\in\Cc^1(G,\Yc)$, then one can show that $\phi\in\LUCl(G,\Yc)$ as follows.  
Since $\LUCl(G,\Yc)$ is invariant under translations to the left, 
it suffices to check the condition in Definition~\ref{aux4}\eqref{aux4_item3} 
at $g_0=\1\in G$. 
By working in an local chart onto a neighborhood $V_0$ of $0\in\gg$ and denoting by $\ast$ the corresponding local multiplication,  
it follows that we have to prove the following assertion: 
There exists a neighborhood $V$ of $0\in V_0$ with the property that  
for every $\epsilon>0$ there exists $\delta>0$ such that 
if $x\in V$ and $y\in\gg$ with $\Vert y\Vert<\delta$, 
then $x\ast y\in V_0$ and $\Vert\phi(x\ast y)-\phi(x)\Vert<\epsilon$, 
where we denote by $\Vert\cdot\Vert$ some norms that define the topologies of $\gg$ and $\Yc$, respectively. 
This assertion follows at once by the mean value theorem, as soon as we have found the neighborhood $V$ of $0\in\gg$ such that 
the closure of $V$ is contained in $V_0$ and 
the differential $\de\phi\colon V\times\gg\to\Yc$ exists and is continuous, 
since then we can shrink $V$ to get 
$\sup\{\Vert\de\phi(x,\cdot)\Vert\mid x\in V\}<\infty$. 
(Recall from \cite[Subsect. 3.2]{Ham82} that the above continuity property of $\de\phi$ ensures that $\de\phi(x,\cdot)\colon\gg\to\Yc$ is linear for every $x\in V$.)
Since $G$ is a Banach-Lie group and $\phi\in\Cc^1(G,\Yc)$, we can use the left trivalization of the tangent bundle of $G$ along with 
the homeomorphism $\gg\to\Lg(G)$, $x\mapsto\gamma_x$, from Examples~\ref{lin_ex}\eqref{lin_ex_item3} below 
to find a neighborhood $V$ of $0\in\gg$ such that $\de\phi$ is continuous on $V\times\gg$, and this completes the proof of the fact that $\phi\in\LUCl(G,\Yc)$.

We now present a version of Taylor's formula suitable for our present purposes. 
See  \cite{Glo02} for a thorough discussion on how to avoid the assumption that functions should take values in a sequentially complete space. 

\begin{proposition}\label{aux6}
Let $G$ be a topological group, 
$\Yc$ be a 
locally convex space, $n\ge1$, and $\phi\in\LUCl^n(G,\Yc)$. 
Then the following assertions hold: 
\begin{enumerate}
\item\label{aux6_item1} 
For every $g\in G$, $X\in\Lg(G)$, and $t\in\RR$ we have 
$$\phi(gX(t))
=\sum_{j=0}^n\frac{t^j}{j!}((D^\lambda_X)^j\phi)(g)+t^n\chi_1(g,X,t)$$
where $\chi_1\colon G\times\Lg(G)\times\RR\to\Yc$ 
is a function such that 
$\lim\limits_{t\to0}\chi_1(g,X,t)=0$. 
\item\label{aux6_item2}
If $X_1,X_2\in\Lg(G)$, then 
$$\phi(gX_1(t)X_2(t))=\phi(g)+t((D^\lambda_{X_1}+D^\lambda_{X_2})\phi)(g)+t\chi_2(g,X_1,X_2,t)$$
where $\chi_2\colon G\times\Lg(G)\times\Lg(G)\times\RR\to\Yc$ 
is a function satisfying the condition  
$\lim\limits_{t\to0}\chi_2(g,X_1,X_2,t)=0$. 
\end{enumerate}
Moreover, for every $g_0\in G$ there exists a neighborhood $V_0$ such that in the above assertions we have both $\lim\limits_{t\to0}\chi_1(g,X,t)=0$  
and $\lim\limits_{t\to0}\chi_2(g,X_1,X_2,t)=0$ uniformly for $g\in V_0$. 
\end{proposition}

\begin{proof}  
Let us define 
$$(\forall t\in\RR)\quad \phi_{g,X}(t)=\phi(gX(t)). $$
Since $\phi_{g,X}\in\Cc^n(\RR,\Yc)$, we get by Taylor's formula 
(see \cite[Prop.~1.17]{Glo02})
$$\begin{aligned}
\phi_{g,X}(t)
&=
\sum_{j=0}^{n-1}\frac{t^j}{j!}\phi_{g,X}^{(j)}(0)
+\frac{t^n}{(n-1)!}\int\limits_0^1(1-s)^{n-1}\phi_{g,X}^{(n)}(ts)\de s \\
&=\sum_{j=0}^{n}\frac{t^j}{j!}\phi_{g,X}^{(j)}(0) +t^n\chi_1(g,X,t)
\end{aligned}$$
where the function 
$$\begin{aligned}
\chi_1(g,X,t)
&=\frac{1}{(n-1)!}\int\limits_0^1\Bigl((1-s)^{n-1}\phi_{g,X}^{(n)}(ts)
-\frac{1}{n}\phi_{g,X}^{(n)}(0)\Bigr)\de s \\
& =\frac{1}{(n-1)!}\int\limits_0^1\Bigl((1-s)^{n-1}((D^\lambda_X)^n\phi)(gX(ts))
-\frac{1}{n}((D^\lambda_X)^n\phi)(g)\Bigr)\de s 
\end{aligned}$$
has the property $\lim\limits_{t\to 0}\chi_1(g,X,t)=0$ uniformly for $g$ in a suitable neighborhood of an arbitrary point in $G$, since $\phi\in\LUCl^n(G,\Yc)$. 

Assertion \eqref{aux6_item2}  can then be obtained by iterating the formula provided by Assertion~\eqref{aux6_item1}. 
Specifically, by using that formula for $X=X_2$ and $n=1$ we get 
$$(\forall g\in G)(\forall t\in\RR)\quad 
\phi(gX_2(t))=\phi(g)+t(D^\lambda_{X_2}\phi)(g)+t\chi_1(g,X_2,t),$$
hence for arbitrary $g\in G$ and $t\in\RR$ we have 
$$ 
\phi(gX_1(t)X_2(t))=\phi(gX_1(t))+t(D^\lambda_{X_2}\phi)(gX_1(t))+t\chi_1(gX_1(t),X_2,t). $$
On the other hand, by Assertion~\eqref{aux6_item1} for $X=X_1$ and $n=1$ we get 
$$
\phi(gX_1(t))=\phi(g)+t(D^\lambda_{X_1}\phi)(g)+t\chi_1(g,X_1,t),$$
again for all $g\in G$ and $t\in\RR$. 
By plugging in this formula in the previous one, we get 
$$
\phi(gX_1(t)X_2(t))
=\phi(g)+t(D^\lambda_{X_1}\phi)(g)+t(D^\lambda_{X_2}\phi)(g) +t\chi_2(g,X_1,X_2,t)$$
where 
$$\chi_2(g,X_1,X_2,t)=
((D^\lambda_{X_2}\phi)(gX_1(t)-(D^\lambda_{X_2}\phi)(g))+\chi_1(g,X_1,t)+\chi_1(gX_1(t),X_2,t).
$$
Now let $g_0\in G$ arbitrary. 
We have proved above that there exists a neighborhood $V_0$ of $g_0$ such that 
$\lim\limits_{t\to0}\chi_1(g,X,t)=0$ uniformly for $g\in V_0$. 
Moreover, since $D^\lambda_{X_2}\phi\in\LUCl(G,\Yc)$, it follows that 
for a suitable neighborhood $V_1$ of $g_0$ 
we have 
$\lim\limits_{t\to0}((D^\lambda_{X_2}\phi)(gX_1(t)-(D^\lambda_{X_2}\phi)(g))=0$ 
uniformly for $g\in V_1$. 
Moreover, we may assume that $V_1U_1\subseteq V_0$ for a suitable neighborhood $U_1$ of $\1\in G$,  
hence $gX_1(t)\in V_0$ for $g\in V_1$ and $t$ in a suitable neighborhood of $0\in\RR$ (depending only on~$U_1$). 
Then we get 
$\lim\limits_{t\to0}\chi_2(g,X_1,X_2,t)=0$ uniformly for $g\in V_1$, and this completes the proof. 
\end{proof}

The next result is obtained by using the method of proof of 
\cite[\S 4, Th.~1]{Nel69} and \cite[Lemma 2.2]{BC73}. 

\begin{lemma}\label{aux8}
Let $G$ be a topological group and 
$\Yc$ be a 
locally convex space. 
If the one-parameter subgroups $X,X_1,X_2\in\Lg(G)$ have the property $X=X_1+X_2$, then 
$$D^\lambda_X\phi
=D^\lambda_{X_1}\phi+D^\lambda_{X_2}\phi$$
for every $\phi\in\LUCl^1(G,\Yc)$.
\end{lemma}

\begin{proof} 
Let us denote 
$$(\forall t\in\RR)\quad g(t)=X_1(t)X_2(t). $$
For each $g_0\in G$ we have to prove the equality
\begin{equation}\label{aux8_proof_eq1}
(D^\lambda_X\phi)(g_0)
=(D^\lambda_{X_1}\phi)(g_0)+(D^\lambda_{X_2}\phi)(g_0).
\end{equation} 
To this end fix an arbitrary continuous seminorm $\vert\cdot\vert$ on $\Yc$ and let $\epsilon>0$ arbitrary. 
Since $D^\lambda_{X_j}\phi\in\Cc(G,\Yc)$ for $j=1,2$, 
there exists a neighborhood $U\in\Vc_G(\1)$ such that 
\begin{equation}\label{aux8_proof_eq2}
(\forall g\in U)\quad 
\vert((D^\lambda_{X_1}+D^\lambda_{X_2})\phi)(g_0g)
-((D^\lambda_{X_1}+D^\lambda_{X_2})\phi)(g_0)\vert<\epsilon.  
\end{equation}
On the other hand, since $\phi\in\LUCl^1(G,\Yc)$, it follows by 
Proposition~\ref{aux6}\eqref{aux6_item2}  
that, by shrinking $U$, we can find $\delta>0$ such that 
for all $x\in g_0U$ and $t\in(-\delta,\delta)\setminus\{0\}$ we have 
\begin{equation}\label{aux8_proof_eq3}
\Bigl\vert\frac{1}{t}(\phi(xg(t))-\phi(x))
-((D^\lambda_{X_1}+D^\lambda_{X_2})\phi)(x)\Bigr\vert
<\frac{\epsilon}{2}. 
\end{equation}
Now let $\delta_1>0$ such that $X(t)\in U_0$ if $-\delta_1\le t\le\delta_1$, 
where $U_0$ is a neighborhood of $\1\in G$ such that $U_0U_0\subseteq U$. 
By using \eqref{aux8_eq1} with uniform convergence on the interval $[-\delta_1,\delta_1]$, we get $n_1\ge 1$ such that 
if $n\ge n_1$ and $-\delta_1\le t\le\delta_1$, 
then $g(t/n)^n\in U$. 
There exists $\delta_2\in(0,\delta_1)$ such that 
if $n= 1,\dots,n_1$ and $-\delta_2\le t\le\delta_2$, 
then $g(t/n)^n\in U$. 
Therefore 
$$g\Bigl(\frac{t}{n}\Bigr)^n\in U
\text{ if } -\delta_2\le t\le\delta_2 
\text{ and }n\ge 1.
$$
Moreover, if $1\le k\le n$ and $\vert t\vert\le\delta_2$, then $\vert(k/n)t\vert\le\vert t\vert\le\delta_2$, 
hence 
$$g\Bigl(\frac{t}{n}\Bigr)^k
=g\Bigl(\frac{(k/n)t}{k}\Bigr)^k
\in U.$$ 
Thus 
$$g\Bigl(\frac{t}{n}\Bigr)^k\in U\text{ if } -\delta_2\le t\le\delta_2 
\text{ and }1\le k\le n.
$$
This allows us to use \eqref{aux8_proof_eq2} and \eqref{aux8_proof_eq3} 
in order to show that if $-\delta_2\le t\le\delta_2$, $t\ne0$,  and $n\ge 1$, then 
\allowdisplaybreaks
\begin{align}
&\Bigl\vert\frac{1}{t}
\Bigl(\phi\Bigl(g_0g\Bigl(\frac{t}{n}\Bigr)^n\Bigr)-\phi(g_0)\Bigr)
-((D^\lambda_{X_1}+D^\lambda_{X_2})\phi)(g_0)\Bigr\vert  \nonumber \\
&\le\frac{1}{n}\sum_{k=1}^n
\Bigl\vert\frac{1}{t/n}
\Bigl(
\phi\Bigl(g_0g\Bigl(\frac{t}{n}\Bigr)^k\Bigr)
-\phi\Bigl(g_0g\Bigl(\frac{t}{n}\Bigr)^{k-1}\Bigr)
\Bigr) 
\hskip-3pt 
-((D^\lambda_{X_1}+D^\lambda_{X_2})\phi)\Bigl(g_0g\Bigl(\frac{t}{n}\Bigr)^{k-1}\Bigr)
\Bigr\vert \nonumber \\
&\hskip10pt +\frac{1}{n}\sum_{k=1}^n
\Bigl\vert 
((D^\lambda_{X_1}+D^\lambda_{X_2})\phi)\Bigl(g_0g\Bigl(\frac{t}{n}\Bigr)^{k-1}\Bigr)
-((D^\lambda_{X_1}+D^\lambda_{X_2})\phi)(g_0)
\Bigr\vert 
\nonumber \\
& <\frac{1}{n}\sum_{k=1}^n\frac{\epsilon}{2}+\frac{1}{n}\sum_{k=1}^n\frac{\epsilon}{2}=\epsilon. \nonumber
\end{align}
On the other hand, for every $t\in\RR$ we have 
$\lim\limits_{n\to\infty} g(t/n)^n=X(t)$ in $G$, by~\eqref{aux8_eq1}. 
Since $\phi\colon G\to\Yc$ is continuous, we then get  
$\lim\limits_{n\to\infty}\phi(g_0g(t/n)^n)=\phi(g_0X(t))$ in~$\Yc$. 
It then follows by the above estimates that if $-\delta_2\le t\le\delta_2$ and $t\ne0$, 
then 
$$\Bigl\vert\frac{1}{t}
(\phi(g_0X(t))-\phi(g_0))
-((D^\lambda_{X_1}+D^\lambda_{X_2})\phi)(g_0)\Bigr\vert\le\epsilon.$$ 
For $t\to 0$ and then $\epsilon\to0$ we get  
$$\vert(D^\lambda_X\phi)(g_0)
-((D^\lambda_{X_1}\phi)(g_0)+(D^\lambda_{X_2}\phi)(g_0)) \vert=0$$
Since $\vert\cdot\vert$ is an arbitrary continuous seminorm on 
the Hausdorff locally convex space~$\Yc$, it follows that
\eqref{aux8_proof_eq1} holds. 
\end{proof}

\begin{theorem}\label{lin_th}
Let $G$ be a topological group such that $\Lg(G)$ has a structure of 
real vector space 
whose scalar multiplication and vector addition 
satisfy the following conditions for all 
$t,s\in{\mathbb R}$ and $X_1,X_2\in\Lg(G)$: 
\begin{eqnarray} 
(t X_1)(s) &=& X_1(ts),   \nonumber \\
(X_1+X_2)(t) &=& \lim\limits_{n\to\infty}
\Bigl(X_1\Bigl(\frac{t}{n}\Bigr)X_2\Bigl(\frac{t}{n}\Bigr)\Bigr)^n, 
\nonumber
\end{eqnarray}
where the convergence is assumed to be uniform on any compact subset of $\RR$.  
Then for every locally convex space $\Yc$ and every 
$\phi\in\LUCl^1(G,\Yc)$ the mapping 
$$D^\lambda\phi\colon\Lg(G)\to\LUCl(G\Yc),\quad X\mapsto D^\lambda_X\phi$$
is linear. 
\end{theorem}

\begin{proof}
It is easily checked that the mapping $D^\lambda\phi$ is $\RR$-homogeneous, 
and the fact that it is additive follows by Lemma~\ref{aux8}. 
\end{proof}

We now establish some continuity properties of the linear mappings provided by the above theorem. 
The proof of the following corollary was suggested by the proof of \cite[Prop. 2]{Mag81}.

\begin{corollary}\label{lin_cont}
Assume the setting of Theorem~\ref{lin_th} along with the following additional hypotheses: 
\begin{enumerate}
\item $\Lg(G)$ is endowed with a Baire topology that is stronger than the compact-open topology and is compatible with the vector space structure. 
\item The locally convex space $\Yc$ is metrizable. 
\end{enumerate}
If $\LUCl(G,\Yc)$ is endowed with the topology of pointwise convergence, 
then for every $\phi\in\LUCl^1(G,\Yc)$ the mapping $D^\lambda\phi\colon\Lg(G)\to\LUCl(G,\Yc)$ is linear and continuous. 
\end{corollary}

\begin{proof}
Let $\phi\in\LUCl^1(G,\Yc)$. 
The linearity of $D^\lambda\phi$ follows by Theorem~\ref{lin_th}.  
To prove the continuity property, let $g\in G$ arbitrary. 
It follows by \eqref{aux4_eq1} that the mapping  
$(D^\lambda\phi)(g)\colon\Lg(G)\to\Yc$ 
is the pointwise limit of a sequence of mappings which are continuous with respect to the compact-open topology of $\Lg(G)$, 
hence are also continuous with respect to the Baire topology mentioned in the statement. 
Since $\Yc$ is metrizable and $\Lg(G)$ is a Baire space, it then follows by \cite[Ch. IX, \S 5, Ex. 20b)]{Bou74} that 
the set of discontinuity points of $(D^\lambda\phi)(g)$ is of the first category, 
hence $(D^\lambda\phi)(g)$ has at least one continuity point. 
Since moreover $\Lg(G)$ is a topological vector space, then it easily follows that the linear mapping $(D^\lambda\phi)(g)$ is continuous throughout $\Lg(G)$. 
\end{proof}

\begin{examples}\label{lin_ex}
\normalfont
Here is a list of special classes of topological groups to which Theorem~\ref{lin_th} applies, 
besides the classical case of the finite-dimensional Lie groups 
(which is covered by several of the following situations):
\begin{enumerate} 

\item\label{lin_ex_item3} Locally exponential Lie groups; 
in particular, the Banach-Lie groups.   
A locally convex Lie group $G$ with the Lie algebra $\gg$ is locally exponential if it has a smooth exponential map $\exp_G\colon\gg\to G$ which is a local diffeomorphism at $0\in\gg$ (\cite[Def.~IV.1.1]{Nee06}).  
For $x\in\gg$ let $\gamma_x\colon\RR\to G$, $\gamma_x(t)=\exp_G(tx)$. 
Then the mapping 
$x\mapsto \gamma_x$
is a homeomorphism $\gg\to\Lg(G)$ and for $x,y\in\gg$ we have $\gamma_{x+y}=\gamma_x+\gamma_y$.  
The right-hand side of the latter equation is defined by \eqref{aux8_eq1}, 
and the corresponding uniform convergence on compact subsets of $\RR$ can be obtained by using Taylor's formula; 
compare \cite[Rem. II.1.8 and Lemma~IV.1.17]{Nee06}. 
Thus $\Lg(G)$ with the compact-open topology is turned into a locally convex space isomorphic to~$\gg$. 

\item\label{lin_ex_item3.5} Mapping groups, and in particular loop groups.  
If $M$ is a compact manifold and $K$ is a finite-dimensional Lie group with the Lie algebra $\kg$, then it follows as a very special case of \cite[Th.~IV.1.12]{Nee06} that the mapping group $\Ci(M,K)$ is a locally exponential Lie group with the Lie algebra $\Ci(M,\kg)$, 
hence we are actually placed in the above situation \eqref{lin_ex_item3}. 
The loop groups are obtained when $M$ is the unit circle. 

\item\label{lin_ex_item3.6} Diffeomorphism groups of compact manifolds. 
If $G$ is the diffeomorphism group of a compact manifold $M$ 
(\cite[Ex. 1.4]{Mil84}, \cite[Ex. II.3.14]{Nee06}), 
then it follows by \cite[Th. 5]{CM70}  
that every continuous one-parameter subgroup of $G$ is smooth. 
The sum of vector fields agrees with the operation defined by~\eqref{aux8_eq1} 
(see for instance \cite[\S 4, Th.~1]{Nel69}), 
hence $\Lg(G)$ can thus be identified with the linear space of all smooth vector fields on~$M$.  

\item\label{lin_ex_item1} Connected locally compact groups;  
more generally, the connected pro-Lie groups.  
A pro-Lie group $G$ is a topological group 
which is isomorphic (as a topological group) to the limit of 
a projective system $\{G_j\}_{j\in J}$ of finite-dimensional Lie groups. 
The continuity property of the functor $\Lg(\cdot)$ established in \cite[Th.~2.25(ii)]{HM07} shows that  $\Lg(G)$ is isomorphic as a topological Lie algebra to the projective limit of the Lie algebras of the groups $G_j$ for $j\in J$; see also \cite[Lemma~2.1]{BC73}. 
We also note that $G$ is isomorphic to a closed subgroup of a direct product of finite-dimensional Lie groups (\cite[Th.~3.39]{HM07}), 
hence in the limits of the type \eqref{aux8_eq1} we have uniform convergence 
on the compact subsets of $\RR$, by the definition of the direct product topology along with the corresponding uniform convergence in the finite-dimensional Lie groups 
(see the situation \eqref{lin_ex_item3} above). 

\item\label{lin_ex_item4} Unitary groups of finite von Neumann algebras. 
Consider a von Neumann algebra with a faithful normal tracial state~$\tau$.  If $G$ is its unitary group endowed with the strong operator topology, then 
it follows by \cite[Th.~3.6]{Bel10} that $\Lg(G)$ with the sum operation given by \eqref{aux8_eq1} is a vector space isomorphic to the space of all skew-symmetric $\tau$-measurable operators on the space of the GNS representation associated with~$\tau$.  

\item\label{lin_ex_item3.4} Direct limits of finite-dimensional Lie groups. 
If $G$ is the limit of a countable direct system of finite-dimensional Lie groups with injective homomorphisms, 
then $\Lg(G)$ is isomorphic as a topological Lie algebra to the corresponding inductive limit of finite-dimensional Lie algebras. 
Moreover, it follows by  
\cite[Prop. 4.6]{Glo05} and its proof that 
for every $X,Y\in\Lg(G)$ there exists $X+Y\in\Lg(G)$ given by \eqref{aux8_eq1} with uniform convergence 
on the compact subsets of $\RR$. 

\item\label{lin_ex_item2} Nilpotent topological groups.  
If $G$ is such a group, then it follows by \cite[Th. 1]{MS75} that 
the topology of any subgroup generated by finitely many elements in $\Lg(G)$ can be refined to a unique topology of a finite-dimensional Lie group. 
Therefore, one can use the situation of finite-dimensional Lie groups 
to see that for every $X,Y\in\Lg(G)$ there exists $X+Y\in\Lg(G)$ given by \eqref{aux8_eq1} with uniform convergence 
on the compact subsets of $\RR$, 
and moreover $\Lg(G)$ is thus turned into a vector space.  
See also \cite[Th.~IV.1.24]{Nee06} for a related result on 2-step nilpotent groups.   

\end{enumerate} 
\end{examples}

We now show that Corollary~\ref{lin_cont} applies and leads to continuity properties in several of the situations mentioned in Examples~\ref{lin_ex}. 

\begin{examples}\label{cont_ex}
\normalfont 
Here we use the same numbering and notation as in Examples~\ref{lin_ex}.  
Continuity properties can be established in the situations \eqref{lin_ex_item3}--\eqref{lin_ex_item4}.
\begin{itemize}

\item[\eqref{lin_ex_item3}] If $G$ is a locally exponential Lie group whose Lie algebra is a Baire space, then $\Lg(G)$ is in particular a Baire space with respect to the compact-open topology. 

\item[\eqref{lin_ex_item3.5}] 
If $G=\Ci(M,K)$, then its Lie algebra $\gg=\Ci(M,\kg)$ is a Fr\'echet space, 
hence we are placed in the above situation.

\item[\eqref{lin_ex_item3.6}] 
If $G$ is the diffeomorphism group of a compact manifold $M$, then we have already seen  that $\Lg(G)$ can be identified with the vector space $\Vc(M)$ of all smooth vector fields on~$M$. 
Moreover, $\Vc(M)$ is a Fr\'echet space and it follows by \cite[Rem. II.5.3, Rem. III.2.5, and Th. III.3.1]{Nee06} that the exponential map $\exp_G\colon\Vc(M)\to G$ 
is smooth.  
In particular, the mapping $\Psi\colon\RR\times\Vc\to G$, $\Psi(t,x)=\exp_G(tx)$, is continuous.  
Then, by using the fact that $\RR$ is locally compact, 
we get by a straightforward reasoning that the mapping $\Vc\to\Cc(\RR,G)$, $x\mapsto\gamma_x(\cdot):=\Psi(\cdot,x)$ 
is continuous with respect to the compact-open topology on $\Cc(\RR,G)$.  
Therefore the mapping 
$x\mapsto \gamma_x$
is continuous from $\Vc(M)$ into $\Lg(G)$,  
and this shows that 
the Fr\'echet topology induced from $\Vc(M)$ on $\Lg(G)$ is stronger than the compact-open topology. 

\item[\eqref{lin_ex_item1}] 
If $G$ is a connected pro-Lie group, then by \cite[Th.~3.12, Prop. 3.8, and Cor. A2.9]{HM07} we get a linear topological isomorphism of $\Lg(G)$ 
onto $\RR^J$ for a suitable set~$J$, so $\Lg(G)$ is a Baire space 
with respect to the compact-open topology by 
\cite[Ch. IX, \S 5, Ex. 16a)]{Bou74}. 

\item[\eqref{lin_ex_item4}] 
If $G$ is the unitary group of a von Neumann algebra with a finite trace~$\tau$, then $\Lg(G)$ with the compact-open topology of is homeomorphic (by the linear isomorphism) to the aforementioned space of measurable operators endowed with the $\tau$-measure topology, which is linear and defined by a complete metric that is invariant under translations; see for instance \cite[Rem. 2.4]{Bel10}. 
Thus $\Lg(G)$ is a Baire space with respect to the compact-open topology. 

\end{itemize}  
So in any of the above cases it follows by Corollary~\ref{lin_cont} that if $\Yc$ is a metrizable locally convex space and $\phi\in\LUCl^1(G,\Yc)$, then $D^\lambda\phi\colon\Lg(G)\to\Yc$ is a linear continuous mapping. 
We thus get a possible approach to \cite[Probl. 13.4]{Nee10}, which in particular asked for conditions ensuring that if $G$ is a Banach-Lie group 
with the Lie algebra~$\gg$, and $\phi\colon G\to\RR$ is a function such that for some $g\in G$ and every $x\in\gg$ the derivative 
$$\de\phi(g)(g.x):=\frac{\de}{\de t}\Bigl\vert_{t=0}\phi(g\exp_G(tx)) $$ 
exists, then the functional $\de\phi(g)\colon\gg\to\RR$ is linear. 
\end{examples}

\begin{remark}
\normalfont 
The method of proof of Lemma~\ref{aux8} also gives the following result: 
Let $G$ be a topological group and $\Yc$ be a locally convex space. 
If $X,X_1,X_2\in\Lg(G)$ have the property that   
\begin{equation*}
(\forall t\in\RR)\quad 
X(t^2) = \lim\limits_{n\to\infty}
\Bigl(X_1\Bigl(\frac{t}{n}\Bigr)
X_2\Bigl(\frac{t}{n}\Bigr)
X_1\Bigl(\frac{t}{n}\Bigr)^{-1}
X_2\Bigl(\frac{t}{n}\Bigr)^{-1}\Bigr)^{n^2} 
\end{equation*}
uniformly on every compact subset of $\RR$, then 
$$D^\lambda_X\phi
=D^\lambda_{X_1}(D^\lambda_{X_2}\phi)-D^\lambda_{X_2}(D^\lambda_{X_1}\phi)$$
for every $\phi\in\LUCl^2(G,\Yc)$.
\end{remark}

\section{Differentiable vectors}

In this section we apply Theorem~\ref{lin_th} and Remark~\ref{lin_cont} to the study of differentiable vectors in group representations. 
We shall use the notation set up in the Introduction, and the main result is Theorem~\ref{final}. 

\begin{setting}\label{set}
\normalfont
Unless otherwise mentioned, we assume the following: 
\begin{itemize}
\item $G$ is a topological group; 
\item $\Yc$ is a locally convex space and $\End(\Yc)$ stands for the space of continuous linear maps on $\Yc$; 
\item $\pi\colon G\to\End(\Yc)$ is a representation  such that the group action 
$$G\times\Yc\to\Yc, (g,y)\mapsto\pi(g)y $$ 
is a continuous mapping; 
\item there exists a neighborhood $V$ of $\1\in G$ such that the set of operators $\pi(V)$ is equi-continuous on $\Yc$.
\end{itemize}
For instance, these assumptions are satisfied if the aforementioned action is continuous and $\Yc$ is a Banach space; see \cite[Lemma 5.2]{Nee10}. 
\end{setting}

\begin{definition}\label{t5}
\normalfont
Let us denote $\Dc_{\de\pi}^0:=\Yc$. 
In addition to the spaces  
$\Dc_{\de\pi}^1$, $\Dc_{\de\pi}^{1,\lin}$, and $\Dc_{\de\pi}^{1,\cont}$ defined in the Introduction, we also define 
inductively for every $k\ge 1$,
$$\Dc_{\de\pi}^{k+1}=\{y\in\Dc_{\de\pi}^1\mid(\forall X\in\Lg(G))\quad \de\pi(X)y\in\Dc_{\de\pi}^k\}. $$
It is clear that $\Dc_{\de\pi}^0\supseteq\Dc_{\de\pi}^1\supseteq\Dc_{\de\pi}^2\supseteq\cdots$, and we define
$$\Dc_{\de\pi}^\infty=\bigcap_{k\ge 1}\Dc_{\de\pi}^k. $$
The spaces 
$$\begin{aligned}
\Ci(\de\pi):=&\bigcap_{X\in\Lg(G)}\Bigl(\bigcap_{k\ge 1}\Dc(\de\pi(X)^k)\Bigr), \\
\Yc_\infty:=&\{y\in\Yc\mid \pi(\cdot)y\in\Ci(G,\Yc)\}
\end{aligned}$$
will also be needed. 
\end{definition}

\begin{remark}
\normalfont 
It is easily seen that  $\Ci(\de\pi)$, $\Yc_\infty$, and $\Dc_{\de\pi}^k$ for $k=0,1,\dots,\infty$ are linear subspaces of $\Yc$ which are invariant under 
the family of operators $\pi(G)$. 
We have 
$$\Yc_\infty\subseteq \Dc_{\de\pi}^\infty\subseteq\Ci(\de\pi)\subseteq\Dc_{\de\pi}^1$$
and the first two of these inclusions are actually equalities if $G$ is a finite-dimensional Lie group, as a consequence of Goodman's theorem; 
see \cite[subsect. 4.4.4]{Wa72} for a broader discussion. 
The space $\Yc_\infty$ for representations of topological groups was introduced in \cite{Bos76}, the definition of $\Ci(\de\pi)$ agrees with \cite[Def. 1]{Mag81} for locally compact groups, and finally the spaces $\Dc_{\de\pi}^k$ for $k=0,1,\dots,\infty$ were introduced in \cite[Def. 3.1]{Nee10} for representations of locally convex Lie groups.
\end{remark}

\begin{lemma}\label{t5.5}
We have 
$\Dc_{\de\pi}^k
=\{y\in\Yc\mid\pi(\cdot)y\in\LUCl^k(G,\Yc)\}$ for every $k=0,1,\dots,\infty$. 
\end{lemma}

\begin{proof} 
For $k=0$ we have to prove that for arbitrary $y\in\Yc$, 
\begin{equation}\label{t5.5_proof_eq1}
 \pi(\cdot)y\in\LUCl(G,\Yc).
 \end{equation}
In fact, for arbitrary $g_0,h\in G$ and $g\in V$ we have 
$$\pi(gg_0h)y-\pi(gg_0)y=\pi(g)\pi(g_0)(\pi(h)y-y).$$
Since the family of operators $\pi(V)$ is equi-continuous on $\Yc$, 
it easily follows by the above equality that for 
every neighborhood $U$ of $0\in\Yc$ there exists a neighborhood 
$V_1$ of $\1\in V$ such that 
for all $h\in V_1$ and $g\in V$ we have $\pi(gg_0h)y-\pi(gg_0)y\in U$. 
This shows that $\pi(\cdot)y$ is left uniformly continuous on the neighborhood $Vg_0$ of $g_0\in V$, and \eqref{t5.5_proof_eq1} is completely proved. 

If $X\in\Lg(G)$, then $y\in\Dc(\de\pi(X))$ if and only if 
the function $\pi(X(\cdot))y$ is differentiable at $0\in\RR$. 
If this the case, then 
\begin{equation}\label{t5.5_proof_eq2}
D^\lambda_X(\pi(\cdot)y)=\pi(\cdot)\de\pi(X)y
\end{equation} 
and now \eqref{t5.5_proof_eq1} shows that $D^\lambda_X(\pi(\cdot)y)\in\LUCl(G,\Yc)$. 
These remarks show that the assertion holds for $k=1$. 
The general case follows by induction on $k$. 
\end{proof}

\begin{theorem}\label{final}
Within Setting~\ref{set} we have $\Dc_{\de\pi}^1=\Dc_{\de\pi}^{1,\lin}$.  
If the following additional conditions are satisfied:
\begin{enumerate} 
\item $\Lg(G)$ is endowed with a Baire topology that is stronger than the compact-open topology; 
\item for every $X_1,X_2\in \Lg(G)$ there exists $X\in\Lg(G)$ such that 
the property $X=X_1+X_2$ holds true, and $\Lg(G)$ is a topological vector space with the vector sum thus defined; 
\item the locally convex space $\Yc$ is metrizable; 
\end{enumerate}
then we have also $\Dc_{\de\pi}^1=\Dc_{\de\pi}^{1,\cont}$. 
\end{theorem}

\begin{proof}
Let $y\in\Dc_{\de\pi}^1$ be arbitrary and denote $\phi:=\pi(\cdot)y$. 

For the first part of the statement we have to show that $y\in\Dc_{\de\pi}^{1,\lin}$. 
To this end let $X,X_1,X_2\in\Lg(G)$ with the property $X=X_1+X_2$. 
Since $y\in\Dc_{\de\pi}^1$, we get $\phi\in\LUCl^1(G,\Yc)$ by Lemma~\ref{t5.5}, and then  
$D^\lambda_X\phi=D^\lambda_{X_1}\phi+D^\lambda_{X_2}\phi\in\LUCl(G,\Yc)$ 
by Lemma~\ref{aux8}. 
By evaluating both sides of this equality at $\1\in G$ 
and by using~\eqref{t5.5_proof_eq2}, we get $\de\pi(X)y=\de\pi(X_1)y+\de\pi(X_2)y$. 
Therefore $y\in\Dc_{\de\pi}^{1,\lin}$. 

For the second part of the statement we have to show that if the additional 
conditions 1.--3.\ are satisfied, then conditions $y\in\Dc_{\de\pi}^{1,\cont}$. 
To this end we just have to use Corollary~\ref{lin_cont} along with \eqref{t5.5_proof_eq2} again evaluated at $\1\in G$.
\end{proof}

\begin{remark}
\normalfont 
We have seen in Example~\ref{cont_ex}\eqref{lin_ex_item3} that the additional conditions 1.--2.\ 
in Theorem~\ref{final} are satisfied if $G$ is a locally exponential Lie group modeled on a locally convex space which is a Baire space; 
for instance if $G$ is a locally exponential Fr\'echet-Lie group. 
On the other hand, these conditions are also satisfied by groups as in Example~\ref{cont_ex}\eqref{lin_ex_item3.6} whose exponential maps may fail to be locally surjective (see e.g., \cite[Warning 1.6]{Mil84}) 
or as in as in Example~\ref{cont_ex}\eqref{lin_ex_item4} 
where it may fail to be locally injective (e.g., unitary groups of von Neumann factors of type $\text{II}_1$; 
see \cite[Cor. 4.4]{Bel10}). 

We now briefly mention a few other special cases of Theorem~\ref{final} and related results that occurred in the earlier literature. 
\begin{enumerate}
\item It follows by \cite[Satz 2.1]{Bos76} that 
if $G$ is a connected pro-Lie group, then $\Yc_\infty\subseteq\Dc_{\de\pi}^{1,\lin}\cap\Dc_{\de\pi}^{1,\cont}$. 
\item It was proved in \cite[Prop. 2]{Mag81} that if $G$ is a a connected locally compact group, $\Yc$ is a Hilbert space, and $\pi$ is a unitary representation, then $\Ci(\de\pi)\subseteq \Dc_{\de\pi}^{1,\lin}\cap \Dc_{\de\pi}^{1,\cont}$. 
\item It was proved in \cite[Lemma 8.1]{Nee10} that if $G$ is a Fr\'echet-Lie group and $\Yc$ is a Fr\'echet space, then 
$\Dc_{\de\pi}^{1,\lin}\subseteq \Dc_{\de\pi}^{1,\cont}$. 
\item The conclusion of our Theorem~\ref{final} was then obtained in \cite[Th. 8.2]{Nee10}  in the case when $G$ is a Banach-Lie group and $\Yc$ is a Banach space. 
\end{enumerate}
\end{remark}

{\bf Acknowledgment.} 
We wish to thank Karl-Hermann Neeb for several useful suggestions. 
The Referee's remarks also helped us to improve the presentation. 
The second-named author acknowledges partial financial support from
CNCSIS - UEFISCDI, grant PNII - IDEI code 1194/2008.


\begin{thebibliography}{999999}

\bibitem[Bel10]{Bel10}
\textsc{D.~Belti\c t\u a}, 
Lie theoretic significance of the measure topologies associated with a finite trace. 
\textit{Forum Math.} \textbf{22} (2010), no.~2, 241--253.

\bibitem[BB09]{BB09}
\textsc{I. Belti\c t\u a, D. Belti\c t\u a}, 
Magnetic pseudo-differential Weyl calculus on nilpotent Lie groups. 
\textit{Ann. Global Anal. Geom.} \textbf{36} (2009), no.~3, 293--322. 

\bibitem[BB11]{BB11}
\textsc{I.~Belti\c t\u a, D.~Belti\c t\u a},  
Continuity of magnetic Weyl calculus. 
{\it J. Funct. Anal.} \textbf{260} (2011), no.~7, 1944--1968. 

\bibitem[Bos76]{Bos76} 
\textsc{H.~Boseck}, 
\"Uber Darstellungen lokal-kompakter topologischer Gruppen.  
\textit{Math. Nachr.} \textbf{74} (1976), 233--251.

\bibitem[BC73]{BC73} 
\textsc{H.~Boseck, G.~Czichowski}, 
Grundfunktionen und verallgemeinerte Funktionen auf topologischen Gruppen. I. 
\textit{Math. Nachr.} \textbf{58} (1973), 215--240. 

\bibitem[BCR81]{BCR81} 
\textsc{H.~Boseck, G.~Czichowski, K.-P.~Rudolph}, 
\textit{Analysis on Topological Groups ---General Lie Theory}.  
Teubner-Texte zur Mathematik, 37. 
BSB B. G. Teubner Verlagsgesellschaft, Leipzig, 1981.

\bibitem[Bou74]{Bou74}
\textsc{N.~Bourbaki}, 
\textit{Topologie G\'en\'erale}. Chap. 5 \`a 10. Nouvelle \'edition.  
Hermann, Paris, 1974. 

\bibitem[Che74]{Che74}
\textsc{P.R.~Chernoff}, 
\textit{Product formulas, nonlinear semigroups, and addition of unbounded operators}. 
Memoirs of the American Mathematical Society, No. 140. American Mathematical Society, Providence, R. I., 1974. 

\bibitem[CM70]{CM70}
\textsc{P.~Chernoff, J. Marsden}, 
On continuity and smoothness of group actions. 
\textit{Bull. Amer. Math. Soc.} \textbf{76} (1970), 1044--1049.

\bibitem[Far75]{Far75}
\textsc{W.G.~Faris}, 
\textit{Self-Adjoint Operators}. 
Lecture Notes in Mathematics, Vol. 433. 
Springer-Verlag, Berlin-Heidelberg-New York, 1975.

\bibitem[Gl\"o02]{Glo02}
\textsc{H.~Gl\"ockner}, 
Infinite-dimensional Lie groups without completeness restrictions. 
In: A. Strasburger, J. Hilgert, K.-H. Neeb, W. Wojtynski (eds.), \textit{Geometry and Analysis on Finite- and Infinite-Dimensional Lie Groups} (B\c edlewo, 2000), Banach Center Publ., 55, Polish Acad. Sci., Warsaw, 2002, pp.~43--59. 

\bibitem[Gl\"o05]{Glo05}
\textsc{H.~Gl\"ockner}, 
Fundamentals of direct limit Lie theory. 
\textit{Compos. Math.} \textbf{141} (2005), no. 6, 1551--1577. 

\bibitem[Ham82]{Ham82}
\textsc{R.S.~Hamilton}, 
The inverse function theorem of Nash and Moser. 
\textit{Bull. Amer. Math. Soc. (N.S.)} \textbf{7} (1982), no.~1, 65--222.

\bibitem[HM07]{HM07}
\textsc{K.H.~Hofmann, S.A.~Morris}, 
\textit{The Lie Theory of Connected Pro-Lie Groups}. 
A Structure Theory for Pro-Lie Algebras, Pro-Lie Groups, 
and Connected Locally Compact Groups. 
EMS Tracts in Mathematics, 2. European Mathematical Society (EMS), 
Z\"urich, 2007.

\bibitem[Mag81]{Mag81}
\textsc{L.~Magnin}, 
Some remarks about $C^{\infty }$ vectors in representations of connected locally compact groups. 
\textit{Pacific J. Math.} \textbf{95} (1981), no.~2, 391--400. 

\bibitem[MS75]{MS75}
\textsc{L.~Magnin, J.~Simon}, 
Lie algebras associated with topological nilpotent groups. 
\textit{Rep. Mathematical Phys.} \textbf{8} (1975), no.~2, 171--180.

\bibitem[Mil84]{Mil84}
\textsc{J. Milnor}, 
Remarks on infinite-dimensional Lie groups. 
In: B.S.~DeWitt and R.~Stora (eds.), 
\textit{Relativit\'e, groupes et topologie. II} (Les Houches, 1983), North-Holland, Amsterdam, 1984, pp. 1007--1057. 

\bibitem[Nee06]{Nee06} 
\textsc{K.-H.~Neeb}, 
Towards a Lie theory of locally convex groups. 
\textit{Japanese J. Math.} \textbf{1} (2006), no.~2, 291--468.

\bibitem[Nee10a]{Nee10} 
\textsc{K.-H.~Neeb}, 
On differentiable vectors for representations of infinite dimensional Lie groups.  
\textit{J. Funct. Anal.} \textbf{259}  (2010),  no.~11, 2814--2855. 

\bibitem[Nee10b]{Nee11} 
\textsc{K.-H.~Neeb}, 
On analytic vectors for unitary representations of infinite dimensional Lie groups.  
\textit{Ann. Inst. Fourier (Grenoble)} (to appear). 

\bibitem[Nel69]{Nel69}
\textsc{E.~Nelson}, 
\textit{Topics in Dynamics. I: Flows}. 
Mathematical Notes. 
Princeton University Press, Princeton, N.J.; University of Tokyo Press, Tokyo,  1969.  

\bibitem[Wa72]{Wa72}
\textsc{G.~Warner}, 
\textit{Harmonic Analysis on Semi-Simple Lie Groups}.~I. 
Die Grundlehren der mathematischen Wissenschaften, Band 188. 
Springer-Verlag, New York-Heidelberg, 1972.


\end{thebibliography}
\end{document}